\documentclass{amsart}

\usepackage{amsfonts,amsthm,hyperref,mathtools}

\newcommand{\F}{\mathbb{F}}
\DeclareMathOperator{\Det}{Det}
\DeclareMathOperator{\IM}{Im}
\DeclareMathOperator{\Tr}{Tr}

\DeclareMathOperator{\tr}{tr}
\DeclareMathOperator{\vc}{vec}

\theoremstyle{plain}
\newtheorem{theorem}{Theorem}[section]
\newtheorem{lemma}[theorem]{Lemma}
\newtheorem{corollary}[theorem]{Corollary}

\theoremstyle{definition}
\newtheorem{definition}[theorem]{Definition}

\theoremstyle{remark}

\begin{document}

\title{Preserving the trace of the Kronecker sum}

\author{Yorick Hardy}
\address[Y. Hardy]{School of Mathematics,\\
 University of the Witwatersrand,\\
 Johannesburg,\\
 Private Bag 3,\\
 Wits 2050,\\
 South Africa}
\email[Y. Hardy]{yorick.hardy@wits.ac.za}
\author{Ajda Fo\v sner}
\address[A. Fo\v sner]{Faculty of Management,\\
 University of Primorska,\\
 Cankarjeva 5,\\
 SI-6000 Koper,\\
 Slovenia}



\date{}

\begin{abstract}
 The aim of this paper is to study linear preservers of the trace of Kronecker sums $A\oplus B$
 and their connection with preservers of determinants of Kronecker products.
 The partial trace and partial determinant play a fundamental role in
 characterizing the preservers of the trace of Kronecker sums and preservers
 of the determinant of Kronecker products respectively.
\end{abstract}

\maketitle

\section{Introduction}

For positive integers $m,n>2$, let $M_n$ be the algebra of all $n\times n$ matrices over some field $\F$
and let $M_{mn} = M_m\otimes M_n = M_m(M_n)$. Here we consider $\mathbb{F}=\mathbb{R}$ or
$\mathbb{F}=\mathbb{C}$ and define $H_n\subset M_n$ to be the Hermitian matrices in $M_n$.
Linear maps preserving properties of Kronecker products of matrices have received considerable
attention in recent years.  Such maps are closely connected to quantum information science
(see, e.g., \cite{AF}). More recently, Ding et.\ al.\ considered linear preservers of determinants
of Kronecker products of Hermitian matrices \cite{ding17}, i.e., linear maps $\phi:H_{mn}\to H_{mn}$
satisfying
\begin{equation*}
 \det(\phi(A\otimes B)) = \det(A\otimes B)
\end{equation*}
where $A$ and $B$ are Hermitian. A few of the results in \cite{ding17} are restricted to the
case when $A$ and $B$ are positive or negative semidefinite matrices. In order to study this problem
more generally, we make use of the identity
\begin{equation*}
 \det\left(e^A\otimes e^B\right) = e^{\tr(A\oplus B)}
\end{equation*}
where $\oplus$ is the Kronecker sum, i.e.,
\begin{equation*}
 A\oplus B := A\otimes I_n + I_m\otimes B,
\end{equation*}
where $I_k$, $k=m,n$ denotes the $k\times k$ identity matrix.
Under exponentiation of Hermitian matrices, the Kronecker sum arises naturally
as the unique map $\oplus: H_m\times H_n\to H_m\otimes H_n$
satisfying
\begin{equation*}
 e^{A\oplus B} = e^A\otimes e^B, \qquad A\in H_m,\,B\in H_n.
\end{equation*}
Moreover, $A\in M_n$ over $\mathbb{C}$ is non-singular
if and only if $A=e^B$ for some $B\in M_n$ \cite[Example 6.2.15]{horn91}.
Thus, in studying linear maps $\psi:M_{mn}\to M_{mn}$ preserving determinants of (non-singular) Kronecker products
\begin{equation*}
 \det(\psi(A\otimes B)) = \det(A\otimes B)
\end{equation*}
of non-singular matrices $A$ and $B$, it suffices to study maps
$\phi:M_{mn}\to M_{mn}$ preserving the trace of Kronecker sums
\begin{equation}
 \label{eq:pr}%
 \tr(\phi(A\oplus B)) = \tr(A\oplus B).
\end{equation}
In this article we will confine our attention to linear maps $\phi$ satisfying
\eqref{eq:pr}.
We denote by $GL_n(\F)$ and $SL_n(\F)$ the general and special linear
groups in $M_n$ respectively. Thus, we are also interested in preservers of the determinant
of Kronecker product in $GL_{mn}(\F)$ in terms of linear preservers of the Kronecker sum.

In what follows, $m,n$ are positive integers. For a positive integer $k$, $I_k$ denotes the $k\times k$
identity matrix, $0_k$ the $k\times k$ zero matrix, and $E_{ij}^{(k)}$, $1\le i,j\le k$, the $k\times k$
matrix whose entries are all equal to zero except for the $(i,j)$-th entry which is equal to one.
As usual, the symbol $\delta _{ij}$ denotes the Kronecker delta, i.e.,
\begin{displaymath}
\delta_{ij} = 
\begin{cases}
1, & \textrm{if $i=j$}~;\\
0, & \textrm{if $i\ne j$}~.
\end{cases}
\end{displaymath}

The partial trace plays a central role in our investigation. Let $A\in M_{mn} = M_m(M_n) = M_m\otimes M_n$.
Then $A$ can be written as a block matrix $A = (A_{ij})_{ij}$ where $A_{ij}\in M_n$ and
$i,j=1,\ldots, m$. In terms of the Kronecker product we write
\begin{equation*}
 A = \sum_{i,j=1}^m E_{i,j}^{(m)}\otimes A_{ij}.
\end{equation*}
The second partial trace ($\tr_2$) maps each $n\times n$ block of $A$ to its trace, i.e.,
$\tr_2:(A_{ij})_{ij}\mapsto (\tr(A_{ij}))_{ij}$, or equivalently
\begin{equation*}
 \tr_2(A) = \sum_{i,j=1}^m E_{i,j}^{(m)}\otimes \tr(A_{ij}) = \sum_{i,j=1}^m \tr(A_{ij})E_{i,j}^{(m)}.
\end{equation*}
Like the trace, the partial trace is a linear operation. Furthermore, the partial trace
preserves the trace
\begin{equation*}
 \tr(A) = \tr(\tr_2(A)).
\end{equation*}
Similarly, we can define the first partial trace ($\tr_1$). This definition provides
\begin{equation*}
 \tr_1(A) = \sum_{j=1}^m A_{jj}.
\end{equation*}
Finally, we note that for any $B\in M_n$
\begin{equation*}
 \tr_1(A(I_m\otimes B))
  = \sum_{j=1}^m A_{jj}B = \tr_1(A)B
\end{equation*}
and similarly for any $C\in M_m$
\begin{equation*}
 \tr_2(A(C\otimes I_n))
  = \tr_2(A)C.
\end{equation*}
The transpose of the matrix $A\in M_n$ will be denoted by $A^T$.
First, we define RT-symmetry, which plays a similar role
in our analysis similar to that of symmetry in matrix analysis.

\section{RT-Symmetry}

Noting that $M_n$ is an $n^2$-dimensional space, in general we have
\begin{equation*}
 \phi(A) = \sum_{j,k,u,v=1}^n \alpha_{jk;uv}E_{jk}^{(n)}AE_{uv}^{(n)}
\end{equation*}
for some $\alpha_{jk;uv}$ in the underlying field. We define the linear
transform $\phi'$ of $\phi$ by
\begin{equation*}
 \phi'(A) := \sum_{j,k,u,v=1}^n \alpha_{uv;jk}E_{jk}^{(n)}AE_{uv}^{(n)}
          \equiv \sum_{j,k,u,v=1}^n \alpha_{jk;uv}E_{uv}^{(n)}AE_{jk}^{(n)}.
\end{equation*}
Clearly, $(\phi')'=\phi$.
Let $\Phi$ be the matrix representing the linear map $\phi$ in the standard basis,
and $\Phi'$ be the matrix representing $\phi'$. Since
\begin{equation*}
 \phi(E_{pq}^{(n)}) = \sum_{j,v=1}^n \alpha_{jp;qv}E_{jv}^{(n)}
\end{equation*}
it follows that
\begin{equation*}
 \Phi = \sum_{j,p,q,v=1}^n\alpha_{jp;qv}E_{jp}^{(n)}\otimes E_{vq}^{(n)}, \quad
 \Phi' = \sum_{q,k,u,p=1}^n\alpha_{qk;up}E_{up}^{(n)}\otimes E_{kq}^{(n)}.
\end{equation*}
Let $P$ denote the perfect shuffle (also known as the vec-permutation matrix) on $\mathbb{F}^n\otimes\mathbb{F}^n$
\cite{henderson81a,vanloan00a}, i.e. $P^T(A\otimes B)P=B\otimes A$. Then
\begin{equation*}
 \Phi^T = P^T\Phi' P.
\end{equation*}
Consequently, $\phi=\phi'$ if and only if $\Phi^T=P^T\Phi P$.
Equivalently, $\phi=\phi'$ if and only if $R(\Phi)^T = R(\Phi^T)$
where $R$ is the rearrangement operator \cite{vLP}.
The rearrangement operator $R$ is linear, and defined by
$R(A\otimes B)=(\vc A)(\vc B)^T$ where $\vc$ is the vec
operator \cite{henderson81a}.

\begin{definition}
 A linear map $\phi:M_n\to M_n$ satisfying $\phi=\phi'$ is said to be \emph{RT-symmetric}.
 If $\phi=-\phi'$ then $\phi$ is said to be \emph{skew RT-symmetric}.
\end{definition}

The following lemma follows immediately from $\phi=\frac12(\phi+\phi')+\frac12(\phi-\phi')$.

\begin{lemma}
 If the underlying field has characteristic not equal to 2, then
 every linear map $\phi:M_n\to M_n$ is the sum of an RT-symmetric map and
 a skew RT-symmetric map.
\end{lemma}

\begin{definition}
 A linear map $\phi:M_n\to M_n$, over $\mathbb{C}$, satisfying $\phi=\overline{\phi'}$ is said to be \emph{RT-Hermitian}.
 If $\phi=-\overline{\phi'}$ then $\phi$ is said to be \emph{skew RT-Hermitian}.
\end{definition}

\section{Linear trace preservers of Kronecker sums}

We may write a linear map $\phi:M_{mn}\to M_{mn}$ in the operator-sum form
\begin{equation*}
 \phi(M) = \sum_{i=1}^r P_iMQ_i
\end{equation*}
for some matrices $P_i,Q_i$, $i=1,\ldots, r$, of the appropriate sizes.
If $\phi$ preserves the trace of a Kronecker sum $M$, then the
cyclic property of the trace yields
\begin{equation*}
 \tr(M) = \tr(\phi(M)) = \sum_{i=1}^r \tr(Q_iP_iM)
\end{equation*}
and so we need only consider preservers of the form
\begin{equation*}
 \phi(M) = \sum_{i=1}^r Q_iP_iM = PM,
\end{equation*}
where
\begin{equation}
 \label{eq:repr}%
 P := \sum_{i=1}^r Q_iP_i
\end{equation}
and the remaining preservers are all obtained by representations \eqref{eq:repr} of $P$.
First we consider maps of the form $\phi(M)=PM$, $M\in M_{mn}$.

\begin{theorem}
 \label{thm:mainpart}%
 Let $\phi:M_{mn}\to M_{mn}$ be a map given by $\phi:M\mapsto PM$ for
 some $P\in M_{mn}$. Then $\tr\phi(A\oplus B) = \tr(A\oplus B)$ for all $A\in M_m$ and $B\in M_n$,
 if and only if
 \begin{equation*}
  \tr_1(P)=\tr_1(I_{mn}) \quad\text{and}\quad \tr_2(P)=\tr_2(I_{mn}).
 \end{equation*}
\end{theorem}

\begin{proof}
 First, let us write $P$ in block matrix form, $P=(P_{kl})$ where
 each $P_{kl}\in M_n$ for $k,l=1,\ldots,m$. In other words,
 \begin{equation*}
  P = \sum_{k,l=1}^m E_{kl}^{(m)}\otimes P_{kl}.
 \end{equation*}
 Since $\phi$ is linear, the map $\phi$ preserves the trace
 of Kronecker sums if and only if $\phi$ preserves traces of
 Kronecker products of the form $E_{ij}^{(m)}\otimes I_n$ and
 of the form $I_m\otimes E_{kl}^{(n)}$. Thus, we have
 \begin{equation*}
  n\delta_{ij} = \tr(E_{ij}^{(m)}\otimes I_n)
               = \tr(\phi(E_{ij}^{(m)}\otimes I_n)) 
 \end{equation*}
 and
 \begin{equation*}
  \tr(\phi(E_{ij}^{(m)}\otimes I_n)) 
               = \sum_{kl=1}^m\delta_{jk}\delta_{il}\tr(P_{kl}) 
               = \tr(P_{ij}).
 \end{equation*}
 It follows that
 \begin{equation*}
  \tr_2(P) = \sum_{k,l=1}^m \tr(P_{kl})E_{kl}^{(m)} = nI_m.
 \end{equation*}
 For Kronecker products of the form $I_m\otimes E_{kl}^{(n)}$, we find
 \begin{equation*}
  m\delta_{kl} = \tr(I_m\otimes E_{kl}^{(n)})
               = \tr(\phi(I_m\otimes E_{kl}^{(n)})),
 \end{equation*}
 where
 \begin{equation*}
  \tr(\phi(I_m\otimes E_{kl}^{(n)}))
               = \sum_{i,j=1}^m \delta_{ij}\tr(P_{ij}E_{kl}^{(n)})
               = \sum_{j=1}^m(P_{jj})_{lk}.
 \end{equation*}
 Consequently,
 \begin{equation*}
  \tr_1(P) = \sum_{i,j=1}^m P_{jj} = mI_n.
 \end{equation*}
 Conversely, suppose that $\tr_1(P) = mI_n$ and $\tr_2(P) = nI_m$. Then
 \begin{align*}
  \tr(\phi(A\oplus B))
   &= \tr(\tr_2(P(A\otimes I_n))) + \tr(\tr_1(P(I_m\otimes B))) \\
   &= \tr(\tr_2(P)A) + \tr(\tr_1(P)B) = n\tr(A) + m\tr(B) = \tr(A\oplus B).
   \qedhere
 \end{align*}
\end{proof}

\begin{corollary}
 Let $\phi:M_{mn}\to M_{mn}$ be a map given by $\phi:M\mapsto PM$ for
 some $P\in M_{mn}$, where
 \begin{equation*}
  P = I_{mn} + \sum_{j=1}^r A_j\otimes B_j.
 \end{equation*}
 Here, $r$ is the tensor rank of $P-I_{mn}$ over $M_m\otimes M_n$
 and $A_j\in M_m$, $B_j\in M_n$ for $j=1,\ldots,r$.
 Then $\tr\phi(A\oplus B) = \tr(A\oplus B)$, if and only if
 $\tr(A_j)=\tr(B_j)=0$ for $j=1,\ldots,r$.
\end{corollary}

\begin{proof}
 By theorem \ref{thm:mainpart} we need only show that $\tr_1(P)=mI_n$ and $\tr_2(P)=nI_m$
 if and only if $\tr(A_j)=\tr(B_j)=0$ for $j=1,\ldots,r$. The proof of $(\Leftarrow)$
 is immediate. For $(\Rightarrow)$, suppose $\tr_1(P)=mI_n$ and $\tr_2(P)=nI_m$.
 It follows that
 \begin{equation*}
  \sum_{j=1}^r \tr(A_j)B_j = 0_n, \qquad
  \sum_{j=1}^r \tr(B_j)A_j = 0_m.
 \end{equation*}
 Since $r$ is the tensor rank of $P-I_{mn}$, the set $\{\,B_1,\ldots,\, B_r\}$ is a linearly
 independent set and $\tr(A_j)=0$ for $j=1,\ldots,r$. Similarly,
 $\tr(B_j)=0$ for $j=1,\ldots,r$.
\end{proof}

As a consequence of Theorem \ref{thm:mainpart}, we have that $\phi:M\mapsto PM$ satisfies
$\tr\phi(A\oplus B) = \tr(A\oplus B)$ if and only if $\tr_1(\phi(I_{mn})) = \tr_1(I_{mn})$ and
$\tr_2(\phi(I_{mn})) = \tr_2(I_{mn})$. In general, this statement is true modulo a traceless
matrix. We note that any linear map $\phi:M_{mn}\to M_{mn}$ can be written in the form
\begin{equation*}
 \phi(M) = M + \sum_{j=1}^r (A_j\otimes C_j)M(B_j\otimes D_j),
\end{equation*}
where $A_j,B_j\in M_m$ and $C_j,D_j\in M_n$ for $j=1,\ldots,r$.
In the following we will use the commutation operation $[A,B]=AB-BA$ corresponding to the Lie product of matrices $A$ and $B$ of the appropriate sizes.

\begin{lemma}
 \label{lem:ident}%
 Let $\phi:M_{mn}\to M_{mn}$ be a linear map given by
 \begin{equation*}
  \phi(M) = M + \sum_{j=1}^r (A_j\otimes C_j)M(B_j\otimes D_j).
 \end{equation*}
 Then $\tr\phi(A\oplus B) = \tr(A\oplus B)$ for all $A\in M_m$ and $B\in M_n$,
 if and only if
 \begin{equation*}
  \tr_1(\phi(I_{mn})-I_{mn})
   = \sum_{j=1}^r \tr(A_jB_j) [C_j,D_j]
 \end{equation*}
 and
 \begin{equation*}
  \tr_2(\phi(I_{mn})-I_{mn})
   = \sum_{j=1}^r \tr(C_jD_j) [A_j,B_j].
 \end{equation*}
\end{lemma}

\begin{proof}
 The linear map $\phi:M_{mn}\to M_{mn}$ can be written in the form
 \begin{equation*}
  \phi(M) = M + \sum_{j=1}^r (A_j\otimes C_j)M(B_j\otimes D_j)
 \end{equation*}
 where $A_j,B_j\in M_m$ and $C_j,D_j\in M_n$ for $j=1,\ldots,r$.
 Since $\tr\phi(A\oplus B) = \tr(A\oplus B)$ if and only if $\tr\phi(A\otimes I_n) = \tr(A\otimes I_n)$
 and $\tr\phi(I_m\otimes B) = \tr(I_m\otimes B)$ for all $A\in M_m$ and $B\in M_n$, we consider these
 two cases separately. In the first case we have
 \begin{equation*}
  \tr\phi(A\otimes I_n) = n\tr(A) + \tr\left(\left(\sum_{j=1}^r \tr(C_jD_j) B_jA_j\right)A\right)
                        = n\tr(A)
 \end{equation*}
 for all $A\in M_m$. This equation holds if and only if
 \begin{align*}
  0_m &= \sum_{j=1}^r \tr(C_jD_j) B_jA_j = \sum_{j=1}^r \tr(C_jD_j) ([B_j,A_j] + A_jB_j) \\
      &= -Q + \sum_{j=1}^r \tr(C_jD_j)A_jB_j \\
      &= -Q + \tr_2\phi(I_{mn}) - \tr_2 I_{mn}
 \end{align*}
 where $[A,B]:=AB-BA$ is the commutator and
 \begin{equation*}
  Q := \sum_{j=1}^r \tr(C_jD_j) [A_j,B_j] = \tr_2\phi(I_{mn}) - \tr_2 I_{mn}
 \end{equation*}
 is traceless (i.e., $\tr(Q)=0$).
 Similarly, the second case yields that $\tr\phi(I_m\otimes B) = \tr(I_m\otimes B)$
 if and only if
 \begin{equation*}
  \sum_{j=1}^r \tr(A_jB_j) [C_j,D_j] = \tr_1\phi(I_{mn}) - \tr_1 I_{mn}.
  \qedhere
 \end{equation*}
\end{proof}

The commutators in this lemma highlight the traceless character. However, the
anti-commutator plays a similar role. Here, the anti-commutator of matrices $A$ and $B$ is given by $[A,B]_+=AB+BA$.
We state the following lemma without proof, which is almost identical to the previous.

\begin{lemma}
 \label{lem:ac}%
 Let $\phi:M_{mn}\to M_{mn}$ be a linear map given by
 \begin{equation*}
  \phi(M) = M + \sum_{j=1}^r (A_j\otimes C_j)M(B_j\otimes D_j).
 \end{equation*}
 Then $\tr\phi(A\oplus B) = \tr(A\oplus B)$ for all $A\in M_m$ and $B\in M_n$,
 if and only if
 \begin{equation*}
  \tr_1(\phi(I_{mn})-I_{mn})
   = \sum_{j=1}^r \tr(A_jB_j) [C_j,D_j]_+
 \end{equation*}
 and
 \begin{equation*}
  \tr_2(\phi(I_{mn})-I_{mn})
   = \sum_{j=1}^r \tr(C_jD_j) [A_j,B_j]_+.
 \end{equation*}
\end{lemma}

Lemma \ref{lem:ident} shows that the partial traces of the identity matrix
must be preserved modulo a traceless matrix. However, this traceless matrix
is not arbitrary but precisely defined in terms of $\phi$. The following
theorem shows that $\phi'$ plays a fundamental role in the characterization
of $\phi$, and provides an succint characterization for RT-symmetric and
skew RT-symmetric maps in the subsequent two corollaries.

\begin{theorem}
 \label{thm:ident}%
 Let $\phi:M_{mn}\to M_{mn}$ be a linear map.
 Then $\tr\phi(A\oplus B) = \tr(A\oplus B)$
 for all $A\in M_m$ and $B\in M_n$ if and only if
 \begin{equation*}
  \tr_1\phi'(I_{mn}) = \tr_1(I_{mn})
  \quad\text{and}\quad
  \tr_2\phi'(I_{mn}) = \tr_2(I_{mn}).
 \end{equation*}
\end{theorem}

\begin{proof}
 Using the representation of $\phi$ from Lemma \ref{lem:ident} provides
 \begin{align*}
  \tr_1(\phi(I_{mn})) &= \tr_1 I_{mn} + \sum_{j=1}^r \tr(A_jB_j) C_jD_j, \\
  \tr_1(\phi'(I_{mn})) &= \tr_1 I_{mn} + \sum_{j=1}^r \tr(B_jA_j) D_jC_j 
 \end{align*}
 and subtracting these two equations yields
 \begin{equation*}
  \tr_1((\phi-\phi')(I_{mn})) = \sum_{j=1}^r \tr(A_jB_j)[C_j,D_j].
 \end{equation*}
 Similarly,
 \begin{equation*}
  \tr_2((\phi-\phi')(I_{mn})) = \sum_{j=1}^r \tr(C_jD_j)[A_j,B_j].
 \end{equation*}
 From Lemma \ref{lem:ident},
 $\tr\phi(A\oplus B) = \tr(A\oplus B)$ for all $A\in M_m$ and $B\in M_n$
 if and only if
 \begin{equation*}
  \tr_1(\phi(I_{mn})-I_{mn}) = \tr_1((\phi-\phi')(I_{mn}))
  \quad\text{and}\quad
  \tr_2(\phi(I_{mn})-I_{mn}) = \tr_2((\phi-\phi')(I_{mn}))
 \end{equation*}
 if and only if
 \begin{equation*}
  \tr_1(\phi'(I_{mn})) = \tr_1(I_{mn})
  \quad\text{and}\quad
  \tr_2(\phi'(I_{mn})) = \tr_2(I_{mn}).
  \qedhere
 \end{equation*}
\end{proof}

\begin{corollary}
 \label{cor:ident}%
 Let $\phi:M_{mn}\to M_{mn}$ be an RT-symmetric map.
 Then $\tr\phi(A\oplus B) = \tr(A\oplus B)$
 for all $A\in M_m$ and $B\in M_n$ if and only if
 \begin{equation*}
  \tr_1\phi(I_{mn}) = \tr_1(I_{mn})
  \quad\text{and}\quad
  \tr_2\phi(I_{mn}) = \tr_2(I_{mn}).
 \end{equation*}
\end{corollary}

\begin{corollary}
 \label{cor:skident}%
 Let $\phi:M_{mn}\to M_{mn}$ be a skew RT-symmetric map.
 Then $\tr\phi(A\oplus B) = \tr(A\oplus B)$
 for all $A\in M_m$ and $B\in M_n$ if and only if
 \begin{equation*}
  \tr_1\phi(I_{mn}) = -\tr_1(I_{mn})
  \quad\text{and}\quad
  \tr_2\phi(I_{mn}) = -\tr_2(I_{mn}).
 \end{equation*}
\end{corollary}

It is straightforward to extend Corollaries \ref{cor:ident} and \ref{cor:skident}
to the RT-Hermitian and skew RT-Hermitian cases since
$\tr_1(\phi'(I_{mn})) = \tr_1(I_{mn})$
if and only if $\tr_1(\overline{\phi'(I_{mn})}) = \tr_1(I_{mn})$.

Now we are ready to consider the connection with the work in \cite{ding17}.
The connection is provided by the exponential map, i.e.,
\begin{equation*}
 \det(e^A\otimes e^B) = e^{\tr(A\oplus B)}.
\end{equation*}

\section{Determinant preservers of Kronecker products}

The condition given in Lemma \ref{lem:ident} implies that we may characterize
a class of determinant preservers of Kronecker products in terms of partial determinants.
However, the relationship between the partial trace and the partial determinant is not
straightforward. If we restrict our attention to matrices over the complex numbers,
$M_{mn}(\mathbb{C})$, then we have \cite{hardy17}
\begin{equation*}
 \Det(e^A\otimes e^B) = e^{\Tr(A\oplus B)} R_{mn}
\end{equation*}

where $\Det(A) := \sqrt[n]{\det(A)}R_n$ and $\Tr(A) := \tr(A)/n$ for $A\in M_n$,
and $R_n$ is the multiplicative group of $n$-th roots of unity in $\mathbb{C}$.
Furthermore, \cite{hardy17} showed that
\begin{equation*}
 \Det_1(e^A\otimes e^B) = e^{\Tr_1(A\oplus B)} R_{m}
 \quad\text{and}\quad
 \Det_2(e^A\otimes e^B) = e^{\Tr_2(A\oplus B)} R_{n}.
\end{equation*}

Let $\Omega_{mn}\subset M_{mn}(\mathbb{C})$
denote the set of matrices in $M_{mn}(\mathbb{C})$ with each eigenvalue
$\lambda$ satisfying $\IM(\lambda)\in(-\pi,\pi]$. Thus we associate with
every non-singular matrix $A$ a unique matrix $M\in\Omega_{mn}$ such
that $A=e^M$.
Let $\phi:M_{mn}(\mathbb{C})\to M_{mn}(\mathbb{C})$ be a linear map and let
$\psi:GL_{mn}(\mathbb{C})\to GL_{mn}(\mathbb{C})$ be the non-linear map
\begin{equation*}
 \psi(e^M) = e^{\phi(M)}.
\end{equation*}
The map is well defined since $M\in\Omega_{mn}$ is uniquely determined for
every matrix in $GL_{mn}(\mathbb{C})$. We have
\begin{equation*}
 \det\psi(e^M) = \det e^{\phi(M)}
               = e^{\tr\phi(M)}
\end{equation*}
so that $\det\psi(e^M)=\det(e^M)$ if and only if $e^{\tr\phi(M)}=e^{\tr M}$.
By linearity of the trace and $\phi$, this holds if and only if $\tr\phi(M)=\tr M$.
Clearly, linear preservers of the trace of Kronecker sums also preserve
the $\Tr$ of Kronecker sums.
Thus, Lemma \ref{lem:ident} provides the following corollary. We use the
same form for $\phi$ as in Lemma \ref{lem:ident}.

\begin{corollary}
 Let $\phi:\Omega_{mn}(\mathbb{C})\to M_{mn}(\mathbb{C})$ be a linear map.
 The map $\psi:GL_{mn}(\mathbb{C})\to GL_{mn}(\mathbb{C})$ given by
 \begin{equation*}
  \psi(e^M) = e^{\phi(M)}
 \end{equation*}
 satisfies $\Det(\psi(A\otimes B)) = \Det(A\otimes B)$ if and only if
 $\Det_1(\psi(I)) = e^{I_n}U R_m$ for $U\in SL_{n}(\mathbb{C})$ and 
 $\Det_2(\psi(I)) = e^{I_m}V R_n$ for $V\in SL_{m}(\mathbb{C})$, where
 \begin{equation*}
   U = \exp\left(\frac1m\sum_{j=1}^r \tr(A_jB_j) [C_j,D_j]\right),\qquad
   V = \exp\left(\frac1n\sum_{j=1}^r \tr(C_jD_j) [A_j,B_j]\right).
 \end{equation*}
\end{corollary}

The matrices $U\in SL_{n}(\mathbb{C})$ and $V\in SL_{m}(\mathbb{C})$ are not arbitrary.
Theorem \ref{thm:ident} and Corollaries \ref{cor:ident} and \ref{cor:skident} provide
a stronger condition, which we present as our final
theorem.

\begin{theorem}
 \label{thm:detthm}%
 Let $\phi:\Omega_{mn}(\mathbb{C})\to M_{mn}(\mathbb{C})$ be
 an RT-symmetric or RT-Hermitian map.
 The map $\psi:GL_{mn}(\mathbb{C})\to GL_{mn}(\mathbb{C})$
 given by
 \begin{equation*}
  \psi(e^M) = e^{\phi(M)}
 \end{equation*}
 satisfies $\Det(\psi(A\otimes B)) = \Det(A\otimes B)$ if and only if
 \begin{equation*}
  \Det_1(\psi(e^{I_{mn}})) = e^{I_n} R_m \quad\text{and}\quad
  \Det_2(\psi(e^{I_{mn}})) = e^{I_m} R_n.
 \end{equation*}
\end{theorem}

\begin{theorem}
 \label{thm:detthm2}%
 Let $\phi:\Omega_{mn}(\mathbb{C})\to M_{mn}(\mathbb{C})$ be a skew RT-symmetric 
 or skew RT-Hermitian map.
 The map $\psi:GL_{mn}(\mathbb{C})\to GL_{mn}(\mathbb{C})$
 given by
 \begin{equation*}
  \psi(e^M) = e^{\phi(M)}
 \end{equation*}
 satisfies $\Det(\psi(A\otimes B)) = \Det(A\otimes B)$ if and only if
 \begin{equation*}
  \Det_1(\psi(e^{I_{mn}})) = e^{-I_n} R_m \quad\text{and}\quad
  \Det_2(\psi(e^{I_{mn}})) = e^{-I_m} R_n.
 \end{equation*}
\end{theorem}


\section*{Funding}

The first author is supported by the National Research Foundation (NRF), South Africa.
This work is based on the research supported in part by the National Research
Foundation of South Africa (Grant Numbers: 105968). Any opinions, findings and
conclusions or recommendations expressed is that of the author(s), and the NRF
accepts no liability whatsoever in this regard.


\bibliographystyle{tfnlm}
\bibliography{ksumlinpreserve}


\end{document}